\documentclass[12pt]{amsart}
\usepackage{amsmath,amssymb,color,cite,verbatim}

\newtheorem{thm}[equation]{Theorem}

\newtheorem{lemma}[equation]{Lemma}
\newtheorem{cor}[equation]{Corollary}

\theoremstyle{definition}

\newcommand{\F}{\mathbb{F}}
\newcommand{\bP}{\mathbb{P}}

\usepackage[pdftex, bookmarks=false]{hyperref}

\title{Connecting two types of representations of a permutation of $\F_q$}

\author{Zhiguo Ding}
\address{
  Hunan Institute of Traffic Engineering,
  Changsha, Hunan 410005 China
}
\email{ding8191@gmail.com}

\date{\today}

\begin{document}

\begin{abstract}

In this paper, we connect two types of representations of a permutation $\sigma$ of the finite field $\F_q$.
One type is algebraic, in which the permutation is represented as the composition of degree-one polynomials and
$k$ copies of $x^{q-2}$, for some prescribed value of $k$. The other type is combinatorial, in which the permutation
is represented as the composition of a degree-one rational function followed by the product of $k$ $2$-cycles on
$\bP^1(\F_q):=\F_q\cup\{\infty\}$, where each $2$-cycle moves $\infty$. We show that, after modding out by obvious
equivalences amongst the algebraic representations, then for each $k$ there is a bijection between the algebraic
representations of $\sigma$ and the combinatorial representations of $\sigma$. We also prove analogous results
for permutations of $\bP^1(\F_q)$. One consequence is a new characterization of the notion of Carlitz rank of 
a permutation on $\F_q$, which we use elsewhere to provide an explicit formula for the Carlitz rank.  Another 
consequence involves a classical theorem of Carlitz, which says that if $q>2$ then the group of permutations of
$\F_q$ is generated by the permutations induced by degree-one polynomials and $x^{q-2}$. Our bijection provides 
a new perspective from which the two proofs of this result in the literature can be seen to arise naturally, 
without requiring the clever tricks that previously appeared to be needed in order to discover those proofs.

\end{abstract}

\thanks{A version of Theorem~\ref{bijection} of the present paper was originally proved as part of a joint project with Michael 
Zieve which resulted in the paper \cite{DZcr}. The author thanks Michael Zieve for encouraging him to publish Theorem~\ref{bijection} 
separately in the present paper. The author thanks the referees and Michael Zieve for suggesting improvements to the statement 
and proof of Theorem~\ref{bijection}, which yielded a stronger result than appeared in the original version of this paper.}

\maketitle

\section{Introduction}

Throughout, we assume that $q>2$ is a prime power and $\F_q$ is the finite field with $q$ elements. Note that permutations of $\F_q$ 
can be viewed as permutations of $\bP^1(\F_q) := \F_q\cup \{\infty\}$ which fixes $\infty$. We will study two completely different 
types of representations of permutations on $\bP^1(\F_q)$. The first type of representations is essentially algebraical, and the second 
type of representations is essentially combinatorial. 

More precisely, fixing a positive integer $k$ and a permutation $\sigma$ of $\bP^1(\F_q)$, we denote by $\mathcal{A}_{\sigma,k}$ 
the set of all representations of $\sigma$ of the form
$$\sigma = \mu(x) \circ x^{q-2} \circ (x-a_k) \circ x^{q-2} \circ (x-a_{k-1}) \circ \cdots \circ x^{q-2} \circ (x-a_1)$$ 
with $a_1, a_2, \dots, a_k \in \F_q$ and $\mu(x) \in \F_q(x)$ of degree one, and similarly we denote by $\mathcal{C}_{\sigma,k}$
the set of all representations of $\sigma$ of the form
$$\sigma = \nu(x) \circ (b_1,\infty) \circ (b_2,\infty) \circ \cdots \circ (b_k,\infty)$$ 
with $b_1, b_2, \dots, b_k \in \F_q$ and $\nu(x) \in \F_q(x)$ of degree one. 

The presentations in the set $\mathcal{A}_{\sigma,k}$ are algebraically nice, since they are compositions of a degree-one rational 
function $\mu(x)$ with the monomial $x^{q-2}$ and monic degree-one polynomials $x-a_i$ with $1\le i\le k$. The presentations in 
the set $\mathcal{C}_{\sigma,k}$ are combinatorially nice, since they are compositions of a degree-one rational function $\nu(x)$ 
with $2$-cycles on $\bP^1(\F_q)$ of the form $(b_i,\infty)$ with $1\le i\le k$.

Although $\mathcal{A}_{\sigma,k}$ and $\mathcal{C}_{\sigma,k}$ are completely different in nature, Theorem~\ref{first} gives a recipe to turn 
any representation in $\mathcal{A}_{\sigma,k}$ into a representation in $\mathcal{C}_{\sigma,k}$, and Theorem~\ref{second} gives a recipe to 
turn any representation in $\mathcal{C}_{\sigma,k}$ into a representation in $\mathcal{A}_{\sigma,k}$. Moreover, due to Theorem~\ref{inverse} 
the maps $\mathcal{F} : \mathcal{A}_{\sigma,k} \to \mathcal{C}_{\sigma,k}$ and $\mathcal{G} : \mathcal{C}_{\sigma,k} \to \mathcal{A}_{\sigma,k}$ 
induced by the above two recipes respectively are inverses to one another. In other words, there exist naturally two inverse 
bijections between $\mathcal{A}_{\sigma,k}$ and $\mathcal{C}_{\sigma,k}$, which are induced by the recipes as illustrated in Theorem~\ref{first} 
and Theorem~\ref{second} respectively, as stated in the following: 

\begin{thm} \label{bijection}
For any positive integer $k$ and any permutation $\sigma$ of $\bP^1(\F_q):= \F_q\cup \{\infty\}$ with $q>2$, there are two natural inverse 
bijections between the set $\mathcal{A}_{\sigma,k}$ of all representations of $\sigma$ of the form
$$\sigma = \mu(x) \circ x^{q-2} \circ (x-a_k) \circ x^{q-2} \circ (x-a_{k-1}) \circ \cdots \circ x^{q-2} \circ (x-a_1)$$ 
with $a_1, a_2, \dots, a_k \in \F_q$ and $\mu(x) \in \F_q(x)$ of degree one and the set $\mathcal{C}_{\sigma,k}$ of all representations 
of $\sigma$ of the form
$$\sigma = \nu(x) \circ (b_1,\infty) \circ (b_2,\infty) \circ \cdots \circ (b_k,\infty)$$ 
with $b_1, b_2, \dots, b_k \in \F_q$ and $\nu(x) \in \F_q(x)$ of degree one. In particular, the finite sets $\mathcal{A}_{\sigma,k}$ 
and $\mathcal{C}_{\sigma,k}$ have the same cardinality.
\end{thm}

There are two easy consequences of our results. The first consequence is a new characterization of the notion of Carlitz rank of 
a permutation on $\F_q$, based on which an explicit formula about Carlitz rank has been obtained in \cite{DZcr}. The second consequence 
involves a classical theorem of Carlitz, which says that if $q>2$ then the group of permutations of $\F_q$ is generated by the permutations 
induced by degree-one polynomials and $x^{q-2}$. Our bijection in Theorem~\ref{bijection} provides a new perspective from which the two 
proofs of this result in the literature can be seen to arise naturally, without requiring the clever tricks that previously appeared to 
be needed in order to discover those proofs.

This paper is organized as follows. In the next section, we prove some results about permutations on $\bP^1(\F_q)$ which will be used in 
our treatment. In section 3 we give a proof for Theorem~\ref{bijection} together with the explicit recipe, which turns an algebraically 
nice representation of a permutation of $\bP^1(\F_q)$ into a combinatorially nice representation of the same permutation and vice versa. 
We conclude in section 4 by illustrating the above mentioned two consequences of our results.

\section{Basic facts}

Let us begin with the following basic observation:

\begin{lemma}\label{basic}
We have $(0,\infty) = x^{-1}\circ x^{q-2}$ as permutations of\/ $\bP^1(\F_q)$. 
\end{lemma}

\begin{proof}
Note that $x^{q-2} = x^{-1}$ for any nonzero element $x\in \F_q$, so the map $x^{-1}\circ x^{q-2}$ fixes each element in $\F_q^*$. 
It is not hard to verify directly that $x^{-1}\circ x^{q-2}$ exchanges the remaining two points $0$ and $\infty$. 
\end{proof}

Next, we need to know the conjugations of a given $2$-cycle $(b,\infty)$ with $b\in \F_q$ by a degree-one polynomial $x-a$ and 
by the monomial $x^{q-2}$ respectively. For this purpose let us give first the following fact:

\begin{lemma}\label{trivial}
The relation $f\circ (u,v) = (f(u),f(v))\circ f$ holds for any injective map $f$ of sets from $X$ to $Y$ and any two distinct $u,v\in X$.
\end{lemma}

\begin{proof}
Both $f\circ (u,v)$ and $(f(u),f(v))\circ f$ send $x$ to $f(x)$ for any element $x\in X\setminus \{u,v\}$, while both of them 
send $u$ to $f(v)$ and send $v$ to $f(u)$.
\end{proof}

\begin{cor}\label{commute}
The following identities hold as permutations of\/ $\bP^1(\F_q)$: 
\begin{itemize}
\item $(x-a)\circ (b,\infty) = (b-a,\infty)\circ (x-a)$ for any  $a,b\in \F_q$,
\item $x^{q-2}\circ (b,\infty) = (b^{q-2},\infty)\circ x^{q-2}$ for any $b\in \F_q$.
\end{itemize}
\end{cor}

\begin{proof}
The results follows from Lemma~\ref{trivial} by taking $X = Y :=\bP^1(\F_q)$, $u:=b$, $v:=\infty$, letting $f(x):= x-a$ and $f(x):= x^{q-2}$ respectively.
\end{proof}

Corollary~\ref{commute} implies particularly that $x^{q-2}\circ (0,\infty) = (0,\infty)\circ x^{q-2}$, which is just equal to $x^{-1}$ by Lemma~\ref{basic}. 
Moreover, it is interesting to note that the composition of any two of the three permutations $x^{-1}, x^{q-2}, (0,\infty)$ of\/ $\bP^1(\F_q)$ is equal 
to the third one. In other words, the subgroup of permutations of $\bP^1(\F_q)$ with $q>2$ generated by $x^{-1}$, $x^{q-2}$, and $(0,\infty)$ is exactly 
the Klein four-group.

After conjugating $(0,\infty)$ by the translation $x+a$, we can obtain the following expression of a $2$-cycle of the form $(a,\infty)$ with $a\in \F_q$:

\begin{cor}\label{2-cycle}
$(a,\infty) = (x^{-1}+a)\circ x^{q-2}\circ (x-a)$ holds for any $a\in \F_q$.
\end{cor}

\begin{proof}
By applications of Corollary~\ref{commute} and Lemma~\ref{basic}, we have 
$$(a,\infty) = (x+a)\circ (0,\infty)\circ (x-a) = (x+a)\circ x^{-1}\circ x^{q-2} \circ (x-a),$$  
which is equal to $(x^{-1}+a)\circ x^{q-2}\circ (x-a)$ obviously. 
\end{proof}

\section{Proof of Theorem~\ref{bijection}}

For a fixed positive integer $k$, let us define maps $F : \F_q^k \to \F_q^k$ and $G : \F_q^k \to \F_q^k$ as follows, which will be used to 
illustrate the recipes in Theorem~\ref{first} and Theorem~\ref{second} respectively.

Define $F : \F_q^k \to \F_q^k$ by sending $(a_1, a_2, \cdots, a_k)$ to $(b_1, b_2, \cdots, b_k)$, where $b_i:= c_{i,i}$ for $1\le i\le k$, 
and for any fixed $1\le i\le k$ we define $c_{i,j}$ with $0\le j\le i$ inductively by $c_{i,0}:= 0$ and $c_{i,j} := c_{i,j-1}^{q-2} + a_{i-j+1}$.

In order to define the map $G : \F_q^k \to \F_q^k$, we need to use the map $\Phi_{\ell} : \F_q^{\ell} \to \F_q^{\ell-1}$ for each $\ell$ with 
$2\le \ell \le k$, which is defined to send $(e_1, e_2, \dots, e_{\ell})$ to $((e_2-e_1)^{q-2}, (e_3-e_1)^{q-2}, \dots, (e_{\ell}-e_1)^{q-2})$. 
For ease of notation, we will drop the subscript $\ell$ and write $\Phi$ for $\Phi_{\ell}$ when the value of $\ell$ is clear from context.

Define $G : \F_q^k \to \F_q^k$ by sending $(b_1, b_2, \dots, b_k)$ to $(a_1, a_2, \dots, a_k)$, where for $1\le i\le k$ we denote by $a_i$ 
the first entry of $\Phi^{i-1}(b_1, b_2, \dots, b_k)$, in which the map $\Phi^{i-1} : \F_q^k \to \F_q^{k-i+1}$ is really the composition 
$\Phi_{k-i+2} \circ \Phi_{k-i+3} \circ \cdots \circ \Phi_k$ with the convention that $\Phi^0$ is the identity map on $\F_q^k$. More precisely, 
if we write $$\Phi^{i-1}(b_1, b_2, \dots, b_k) = (d_{1,i-1}, d_{2,i-1}, \dots, d_{k-i+1,i-1})$$ for $1\le i\le k$, then by definition we have
$$(a_1, a_2, \dots, a_k) := (d_{1,0}, d_{1,1}, \dots, d_{1,k-1}).$$

\begin{thm}\label{inverse}
The maps $F : \F_q^k \to \F_q^k$ and $G : \F_q^k \to \F_q^k$ are inverses of one another. In particular, both $F$ and $G$ are bijections on $\F_q^k$.
\end{thm}

\begin{proof}
It suffices to show that $G \circ F$ is the identity map on $\F_q^k$, since then $F$ is injective and $G$ is surjective, but since each 
of $F$ and $G$ is a map on $\F_q^k$, it follows that $F$ and $G$ are bijective, whence $F$ and $G$ are inverse bijections since $G \circ F$ 
is the identity map on $\F_q^k$ again.

Now, let us show that $G \circ F$ is the identity map on $\F_q^k$. Given $a_1, a_2, \dots, a_k$ in $\F_q$, recall that for $1 \le i \le k$ 
the elements $c_{i,j}$ with $0\le j\le i$ are defined inductively by $c_{i,0} = 0$ and $c_{i,j} = c_{i,j-1}^{q-2} + a_{i-j+1}$. Put $b_j = c_{j,j}$ 
for $1 \le j \le k$, so that by the definition of $F$ we get $$F(a_1, a_2, \dots, a_k) = (b_1, b_2, \dots, b_k).$$ For $1\le i\le k$ let us write 
$$\Phi^{i-1} (b_1, b_2, \dots, b_k) = (d_{1,i-1}, d_{2,i-1}, \dots, d_{k-i+1,i-1}),$$ then by the definition of $G$ we have
$$G(b_1, b_2, \dots, b_k) = (d_{1,0}, d_{1,1}, \dots, d_{1,k-1}).$$
We will show by induction on $j$ that $d_{i,j} = c_{i+j,i}$ holds for all $i, j$ with $0 \le j \le k-1$ and $1 \le i \le k-j$. The base case $j=0$ 
says that $d_{i,0} = c_{i,i}$ for $1\le i\le k$, which is true since both sides equal $b_i$. Inductively, if $1 \le j \le k-1$, then for each 
$1\le i\le k-j$ we have
\begin{align*}
  & d_{i,j} = (d_{i+1,j-1} - d_{1,j-1})^{q-2} = (c_{i+j,i+1} - c_{j,1})^{q-2} \\
= & (c_{i+j,i+1} - a_j)^{q-2} = (c_{i+j,i}^{q-2})^{q-2} = c_{i+j,i},
\end{align*} 
which concludes the induction. Thus 
\begin{align*}
  & G\circ F (a_1, a_2, \dots, a_k) = G(b_1, b_2, \dots, b_k) \\
= & (d_{1,0}, d_{1,1}, \dots, d_{1,k-1}) = (c_{1,1}, c_{2,1}, \dots, c_{k,1}) = (a_1, a_2, \dots, a_k). 
\end{align*}
So $G \circ F$ is the identity map on $\F_q^k$, whence as explained above it follows 
that $F$ and $G$ are inverse bijections on $\F_q^k$. 
\end{proof}

\begin{thm}\label{first}
For any $a_1, a_2, \dots, a_k$ in $\F_q$, denote $$(b_1, b_2, \cdots, b_k) := F(a_1, a_2, \cdots, a_k),$$ then 
\begin{align*}
  & x^{q-2} \circ (x-a_k) \circ x^{q-2} \circ (x-a_{k-1}) \circ \cdots \circ x^{q-2} \circ (x-a_1) \\
= & \nu(x) \circ (b_1,\infty) \circ (b_2,\infty) \circ \cdots \circ (b_k,\infty),
\end{align*}
where $$\nu(x) := x^{-1} \circ (x-a_k) \circ x^{-1} \circ (x-a_{k-1}) \circ \cdots \circ x^{-1} \circ (x-a_1)$$ 
is a degree-one rational function in $\F_q(x)$. 
\end{thm}

\begin{proof}
We prove it by induction on $k$. The base case $k=1$ says $$x^{q-2} \circ (x-a_1) = x^{-1} \circ (x-a_1) \circ (b_1,\infty)$$ 
where $b_1 = F(a_1) = a_1$. This identity follows directly from Corollary~\ref{2-cycle}. Inductively, suppose $k>1$ and $a_1, a_2, \dots, a_k$ 
are in $\F_q$. By definition, for $1\le i\le k$ we have $c_{i,0} = 0$ and $c_{i,j} = c_{i,j-1}^{q-2} + a_{i-j+1}$ with $1\le j\le i$. Write 
$b_i = c_{i,i}$ for $1\le i\le k$, so that $$(b_1, b_2, \dots, b_k) = F(a_1, a_2, \dots, a_k).$$ For ease of expression, let us denote 
$$\lambda_i := x^{-1} \circ (x-a_k) \circ x^{q-2} \circ (x-a_{k-1}) \circ x^{q-2} \circ \cdots \circ x^{q-2} \circ (x-a_i)$$ for $1\le i\le k$, 
and denote $$\rho_i := x^{q-2} \circ (x-a_i) \circ x^{q-2} \circ (x-a_{i-1}) \circ \cdots \circ x^{q-2} \circ (x-a_1)$$ for $0\le i\le k$ where 
by convention $\rho_0 = x$. Write 
$$\sigma := x^{q-2} \circ (x-a_k) \circ x^{q-2} \circ (x-a_{k-1}) \circ \cdots \circ x^{q-2} \circ (x-a_1).$$ Hence we have
\begin{align*}
\sigma = \rho_k = & x^{q-2} \circ (x-a_k) \circ \rho_{k-1} \\
= & x^{-1} \circ (0,\infty) \circ (x-a_k) \circ \rho_{k-1} \\
= & x^{-1} \circ (x-a_k) \circ (a_k,\infty) \circ \rho_{k-1} \\
= & \lambda_k \circ (c_{k,1},\infty) \circ \rho_{k-1}. 
\end{align*}
We claim that $$\lambda_{i+1} \circ (c_{k,k-i},\infty) \circ \rho_i = \lambda_i \circ (c_{k,k-i+1},\infty) \circ \rho_{i-1}$$ 
for $1\le i\le k-1$. Indeed, for each $i$ with $1\le i\le k-1$, we have 
\begin{align*}
  & (c_{k,k-i},\infty) \circ \rho_i \\
= & (c_{k,k-i},\infty) \circ x^{q-2} \circ (x-a_i) \circ \rho_{i-1} \\
= & x^{q-2} \circ (c_{k,k-i}^{q-2},\infty) \circ (x-a_i) \circ \rho_{i-1} \\
= & x^{q-2} \circ (x-a_i) \circ (c_{k,k-i+1},\infty) \circ \rho_{i-1}. 
\end{align*}
Thus by prepending $\lambda_{i+1}$ to these permutations we get 
\begin{align*}
  & \lambda_{i+1} \circ (c_{k,k-i},\infty) \circ \rho_i \\
= & \lambda_{i+1} \circ x^{q-2} \circ (x-a_i) \circ (c_{k,k-i+1},\infty) \circ \rho_{i-1} \\
= & \lambda_i \circ (c_{k,k-i+1},\infty) \circ \rho_{i-1},
\end{align*}
which concludes the proof of the claim. So we have 
$$\sigma = \lambda_k \circ (c_{k,1},\infty) \circ \rho_{k-1} = \lambda_{k-1} \circ (c_{k,2},\infty) \circ \rho_{k-2} 
= \cdots = \lambda_1 \circ (c_{k,k},\infty) \circ \rho_0,$$ which by definition is equal to 
$$x^{-1} \circ (x-a_k) \circ x^{q-2} \circ (x-a_{k-1}) \circ x^{q-2} \circ \cdots \circ x^{q-2} \circ (x-a_1) \circ (b_k,\infty).$$ 
By inductive hypothesis, we have
\begin{align*}
  & x^{q-2} \circ (x-a_{k-1}) \circ x^{q-2} \circ (x-a_{k-2}) \circ \cdots \circ x^{q-2} \circ (x-a_1) \\
= & x^{-1} \circ (x-a_{k-1}) \circ x^{-1} \circ (x-a_{k-2}) \circ \cdots \circ x^{-1} \circ (x-a_1) \circ \\
  & \circ (b_1,\infty) \circ (b_2,\infty) \circ \cdots \circ (b_{k-1},\infty).
\end{align*} 
Therefore, we get 
\begin{align*}
\sigma = & x^{-1} \circ (x-a_k) \circ x^{-1} \circ (x-a_{k-1}) \circ \cdots \circ x^{-1} \circ (x-a_1) \circ \\
 & \circ (b_1,\infty) \circ (b_2,\infty) \circ \cdots \circ (b_k,\infty),
\end{align*}
 which concludes the induction and hence concludes the proof.
\end{proof}

\begin{thm}\label{second}
For any $b_1, b_2, \dots, b_k$ in $\F_q$, denote $$(a_1, a_2, \dots, a_k) := G(b_1, b_2, \dots, b_k),$$ then 
\begin{align*}
  & (b_1, \infty) \circ (b_2,\infty) \circ \cdots \circ (b_k,\infty) \\
= & \mu(x) \circ x^{q-2} \circ (x-a_k) \circ x^{q-2} \circ (x-a_{k-1}) \circ \cdots \circ x^{q-2} \circ (x-a_1), 
\end{align*}
where $$\mu(x) := (x+a_1) \circ x^{-1} \circ (x+a_2) \circ x^{-1} \circ \cdots \circ (x+a_k) \circ x^{-1}$$ 
is a degree-one rational function in $\F_q(x)$.
\end{thm}

\begin{proof}
Let us prove it by induction on $k$. The base case $k = 1$ says 
$$(b_1,\infty) = (x+a_1) \circ x^{-1} \circ x^{q-2} \circ (x-a_1)$$ 
where $a_1 = G(b_1) = b_1$, which is true by Corollary~\ref{2-cycle}. For the inductive step, assume that $k>1$ and $b_1, b_2, \dots, b_k$ 
are in $\F_q$. Write $$\Phi^{i-1}(b_1, b_2, \dots, b_k) = (d_{1,i-1}, d_{2,i-1}, \dots, d_{k-i+1,i-1})$$ for $1\le i\le k$, so that 
$$(a_1, a_2, \dots, a_k) = G(b_1, b_2, \dots, b_k) = (d_{1,0}, d_{1,1}, \dots, d_{1,k-1}).$$ For ease of expression, let us define 
$$\gamma_i := (x+b_1) \circ x^{-1} \circ (d_{1,1},\infty) \circ (d_{2,1},\infty) \circ \cdots \circ (d_{i-1,1},\infty) \circ x^{q-2} \circ (x-b_1)$$ 
for $1 \le i \le k$, and define $$\theta_i := (b_{i+1},\infty) \circ (b_{i+2},\infty) \circ \dots \circ (b_k,\infty)$$ for $0 \le i \le k$, 
where by convention $\gamma_1 = (x+b_1) \circ x^{-1} \circ x^{q-2} \circ (x-b_1)$ 
and $\theta_k = x$. Denote $$\tau := (b_1,\infty) \circ (b_2,\infty) \circ \cdots \circ (b_k,\infty).$$ Hence by Corollary~\ref{2-cycle} we have 
$$\tau = \theta_0 = (b_1,\infty) \circ \theta_1 = (x+b_1) \circ x^{-1} \circ x^{q-2} \circ (x-b_1) \circ \theta_1 = \gamma_1 \circ \theta_1.$$ 
We claim that $\gamma_i \circ \theta_i = \gamma_{i+1} 
\circ \theta_{i+1}$ for $1\le i\le k-1$. Indeed, for each $i$ with $1\le i\le k-1$, we have 
\begin{align*}
  & x^{q-2} \circ (x-b_1) \circ \theta_i \\
= & x^{q-2} \circ (x-b_1) \circ (b_{i+1},\infty) \circ \theta_{i+1} \\
= & x^{q-2} \circ (b_{i+1}-b_1,\infty) \circ (x-b_1) \circ \theta_{i+1} \\
= & (d_{i,1},\infty) \circ x^{q-2} \circ (x-b_1) \circ \theta_{i+1}. 
\end{align*}
Thus by prepending $$(x+b_1) \circ x^{-1} \circ (d_{1,1},\infty) \circ (d_{2,1},\infty) \circ \cdots \circ (d_{i-1,1},\infty)$$ to these
permutations, we get 
\begin{align*}
\gamma_i \circ \theta_i = & (x+b_1) \circ x^{-1} \circ (d_{1,1},\infty) \circ (d_{2,1},\infty) \circ \cdots \\
  & \circ (d_{i-1,1},\infty) \circ x^{q-2} \circ (x-b_1) \circ \theta_i \\
= & (x+b_1) \circ x^{-1} \circ (d_{1,1},\infty) \circ (d_{2,1},\infty) \circ \cdots \\
  & \circ (d_{i-1,1},\infty) \circ (d_{i,1},\infty) \circ x^{q-2} \circ (x-b_1) \circ \theta_{i+1} \\
= & \gamma_{i+1} \circ \theta_{i+1}, 
\end{align*}
which concludes the proof of the claim. So we have 
$$\tau = \gamma_1 \circ \theta_1 = \gamma_2 \circ \theta_2 = \cdots = \gamma_k \circ \theta_k$$ which by definition is equal to
$$(x+b_1) \circ x^{-1} \circ (d_{1,1},\infty) \circ (d_{2,1},\infty) \circ \cdots \circ (d_{k-1,1},\infty) \circ x^{q-2} \circ (x-b_1).$$ 
By inductive hypothesis, if we denote $$(\widehat{a_1}, \widehat{a_2}, \dots, \widehat{a_{k-1}}) := G(d_{1,1}, d_{2,1}, \dots, d_{k-1,1}),$$ 
then 
\begin{align*}
  & (d_{1,1},\infty) \circ (d_{2,1},\infty) \circ \cdots \circ (d_{k-1,1},\infty) \\
= & (x+\widehat{a_1}) \circ x^{-1} \circ (x+\widehat{a_2}) \circ x^{-1} \circ \cdots \circ (x+\widehat{a_{k-1}}) \circ x^{-1} \circ \\
  & \circ x^{q-2} \circ (x-\widehat{a_{k-1}}) \circ x^{q-2} \circ (x-\widehat{a_{k-2}}) \circ \cdots \circ x^{q-2} \circ (x-\widehat{a_1}). 
\end{align*}
Note that for $1\le i\le k-1$ we have $\widehat{a_i} = a_{i+1}$ since $$\Phi^{i-1}(d_{1,1}, d_{2,1}, \dots, d_{k-1,1}) = \Phi^i(b_1, b_2, \dots, b_k).$$ 
Therefore, we have
\begin{align*}
\tau = & (x+b_1) \circ x^{-1} \circ (x+a_2) \circ x^{-1} \circ (x+a_3) \circ x^{-1} \circ \cdots \\
       & \circ (x+a_k) \circ x^{-1} \circ x^{q-2} \circ (x-a_k) \circ x^{q-2} \circ (x-a_{k-1}) \circ \cdots \\
       & \circ x^{q-2} \circ (x-a_2) \circ x^{q-2} \circ (x-b_1),
\end{align*}
which concludes the induction and the proof since $b_1 = a_1$. 
\end{proof}

We remark that Theorem~\ref{first} gives a recipe for turning an algebraical presentation of a permutation on $\bP^1(\F_q)$ of the form 
$$\mu(x) \circ x^{q-2} \circ (x-a_k) \circ x^{q-2} \circ (x-a_{k-1}) \circ \cdots \circ x^{q-2} \circ (x-a_1)$$ into a combinatorial presentation of the same 
permutation of the form 
$$\nu(x) \circ (b_1, \infty) \circ (b_2,\infty) \circ \cdots \circ (b_k,\infty),$$ 
and conversely Theorem~\ref{second} gives a recipe for turning a combinatorial presentation of a permutation on $\bP^1(\F_q)$ of the form 
$$\nu(x) \circ (b_1, \infty) \circ (b_2,\infty) \circ \cdots \circ (b_k,\infty)$$ 
into an algebraical presentation of the same permutation of the form
$$\mu(x) \circ x^{q-2} \circ (x-a_k) \circ x^{q-2} \circ (x-a_{k-1}) \circ \cdots \circ x^{q-2} \circ (x-a_1),$$ 
where $a_1, a_2 \dots, a_k, b_1, b_2 \dots, b_k$ are elements in $\F_q$ and $\mu(x), \nu(x)$ are degree-one rational functions in $\F_q(x)$. 

In other words, for any fixed permutation $\sigma$ of\/ $\bP^1(\F_q)$, there are two natural maps $\mathcal{F} : \mathcal{A}_{\sigma,k} \to \mathcal{C}_{\sigma,k}$
and $\mathcal{G} : \mathcal{C}_{\sigma,k} \to \mathcal{A}_{\sigma,k}$ induced by the recipes illustrated in Theorem~\ref{first} and Theorem~\ref{second} respectively, 
where $\mathcal{A}_{\sigma,k}$ is the set of all representations of $\sigma$ of the form 
$$\sigma = \mu(x) \circ x^{q-2} \circ (x-a_k) \circ x^{q-2} \circ (x-a_{k-1}) \circ \cdots \circ x^{q-2} \circ (x-a_1)$$ 
with $a_1, a_2, \dots, a_k \in \F_q$ and $\mu(x) \in \F_q(x)$ of degree one, and $\mathcal{C}_{\sigma,k}$ is the set of all representations of $\sigma$ 
of the form $$\sigma = \nu(x) \circ (b_1,\infty) \circ (b_2,\infty) \circ \cdots \circ (b_k,\infty)$$ 
with $b_1, b_2, \dots, b_k \in \F_q$ and $\nu(x) \in \F_q(x)$ of degree one.

Now we are ready to prove Theorem~\ref{bijection}, which asserts the natural maps $\mathcal{F} : \mathcal{A}_{\sigma,k} \to \mathcal{C}_{\sigma,k}$
and $\mathcal{G} : \mathcal{C}_{\sigma,k} \to \mathcal{A}_{\sigma,k}$ induced by the recipes in Theorem~\ref{first} and Theorem~\ref{second} respectively 
are inverses to one another.

\begin{proof}[Proof of Theorem~\ref{bijection}]
It suffices to show that both $\mathcal{G} \circ \mathcal{F}$ and $\mathcal{F} \circ \mathcal{G}$ are the identity maps, where 
$\mathcal{F} : \mathcal{A}_{\sigma,k} \to \mathcal{C}_{\sigma,k}$ and $\mathcal{G} : \mathcal{C}_{\sigma,k} \to \mathcal{A}_{\sigma,k}$ are the natural maps 
induced by Theorem~\ref{first} and Theorem~\ref{second} respectively.

Given any algebraic representation of $\sigma$ in the set $\mathcal{A}_{\sigma,k}$ of the form 
$$\sigma = \mu(x) \circ x^{q-2} \circ (x-a_k) \circ x^{q-2} \circ (x-a_{k-1}) \circ \cdots \circ x^{q-2} \circ (x-a_1)$$ 
with $a_1, a_2, \dots, a_k \in \F_q$ and $\mu(x) \in \F_q(x)$ of degree one, by definition $\mathcal{F}$ sends it to a combinatorial representation 
of $\sigma$ in $\mathcal{C}_{\sigma,k}$ of the form $$\sigma = \nu(x) \circ (b_1,\infty) \circ (b_2,\infty) \circ \cdots \circ (b_k,\infty)$$ 
with $b_1, b_2, \dots, b_k \in \F_q$ and $\nu(x) \in \F_q(x)$ of degree one, in which $$(b_1, b_2, \cdots, b_k) := F(a_1, a_2, \cdots, a_k),$$
$$\nu_0(x) := x^{-1} \circ (x-a_k) \circ x^{-1} \circ (x-a_{k-1}) \circ \cdots \circ x^{-1} \circ (x-a_1),$$ and $\nu(x) := \mu(x) \circ \nu_0(x)$. 
This combinatorial representation is sent by $\mathcal{G}$ to another algebraic representation of $\sigma$ in $\mathcal{A}_{\sigma,k}$ of the form
$$\sigma = \mu^*(x) \circ x^{q-2} \circ (x-a_k^*) \circ x^{q-2} \circ (x-a_{k-1}^*) \circ \cdots \circ x^{q-2} \circ (x-a_1^*)$$ 
with $a_1^*, a_2^*, \dots, a_k^* \in \F_q$ and $\mu^*(x) \in \F_q(x)$ of degree one, in which $$(a_1^*, a_2^*, \dots, a_k^*) := G(b_1, b_2, \dots, b_k),$$ 
$$\mu_0(x) := (x+a_1^*) \circ x^{-1} \circ (x+a_2^*) \circ x^{-1} \circ \cdots \circ (x+a_k^*) \circ x^{-1},$$ and $\mu^*(x) := \nu(x) \circ \mu_0(x)$.
By Theorem~\ref{inverse}, the $k$-tuple 
$$(a_1^*, a_2^*, \dots, a_k^*) = G(b_1, b_2, \dots, b_k) = G\circ F (a_1, a_2, \dots, a_k)$$ equals $(a_1, a_2, \dots, a_k)$, i.e., we have $a_i^* = a_i$ 
for any $1\le i\le k$. Thus $\mu_0(x)$ and $\nu_0(x)$ are inverse to one another, which implies that 
$$\mu^*(x) = \nu(x) \circ \mu_0(x) = \mu(x) \circ \nu_0(x) \circ \mu_0(x) = \mu(x).$$ Hence, $\mathcal{G} \circ \mathcal{F}$ is the identity map 
on the set $\mathcal{A}_{\sigma,k}$. Similarly, we can show that $\mathcal{F} \circ \mathcal{G}$ is the identity map on $\mathcal{C}_{\sigma,k}$. 
Therefore, the natural maps $\mathcal{F} : \mathcal{A}_{\sigma,k} \to \mathcal{C}_{\sigma,k}$ and $\mathcal{G} : \mathcal{C}_{\sigma,k} \to \mathcal{A}_{\sigma,k}$ 
induced by Theorem~\ref{first} and Theorem~\ref{second} respectively are inverse bijections.
\end{proof}

\section{Consequences}

Note that any permutation of\/ $\F_q$ can be extended uniquely to a permutation of\/ $\bP^1(\F_q)$. 
Hence permutations of\/ $\F_q$ can be regarded as permutations of\/ $\bP^1(\F_q)$ in this sense. By application of Theorem~\ref{bijection} to 
permutations of\/ $\F_q$, we get the following result:

\begin{cor}\label{F_q}
Suppose $f$ is a permutation of $\F_q$, and then extends $f$ to a permutation of $\bP^1(\F_q)$ which fixes $\infty$. Then in any representation 
$$f = \mu(x) \circ x^{q-2} \circ (x-a_k) \circ x^{q-2} \circ (x-a_{k-1}) \circ \cdots \circ x^{q-2} \circ (x-a_1)$$ 
with $a_1, a_2, \dots, a_k \in \F_q$ and $\mu(x) \in \F_q(x)$ of degree one, the rational function $\mu(x)$ must be a degree-one polynomial in $\F_q[x]$. 
Therefore, the set of all representations of $f$ in the form of 
$$\mu(x) \circ x^{q-2} \circ (x-a_k) \circ x^{q-2} \circ (x-a_{k-1}) \circ \cdots \circ x^{q-2} \circ (x-a_1)$$ 
with $a_1, a_2, \dots, a_k \in \F_q$ and $\mu(x) \in \F_q[x]$ of degree one is naturally bijective to the set of all representations of $f$ in the form of 
$$\nu(x) \circ (b_1,\infty) \circ (b_2,\infty) \circ \cdots \circ (b_k,\infty)$$ 
with $b_1, b_2, \dots, b_k \in \F_q$ and $\nu(x) \in \F_q(x)$ of degree one. 
\end{cor}

\begin{proof}
Note that the degree-one rational function $\mu(x) \in \F_q(x)$ in any representation of the extended permutation $f$ on $\bP^1(\F_q)$ in the form of
$$f = \mu(x) \circ x^{q-2} \circ (x-a_k) \circ x^{q-2} \circ (x-a_{k-1}) \circ \cdots \circ x^{q-2} \circ (x-a_1)$$ 
fixes the point $\infty$, since all of $f$, $x^{q-2}$, and $(x-a_i)$ fix $\infty$. Hence the first assertion follows from the fact that the degree-one 
rational functions which fix $\infty$ are precisely the degree-one polynomials. The second assertion is then a direct consequence of Theorem~\ref{bijection}.
\end{proof}

The following is a classical result on permutations of $\F_q$, which was posed as a question by Straus, and was first proved by Carlitz \cite{Carlitz} 
and then proved again by Zieve \cite{Z}: 

\begin{thm}\label{Carlitz}
If $q>2$ is a prime power, then every permutation of\/ $\F_q$ is a composition of $x^{q-2}$ and degree-one polynomials over\/ $\F_q$.
\end{thm}

Inspired by Theorem~\ref{Carlitz}, Aksoy et al.\ \cite{ACMT} introduced the notion of the \emph{Carlitz rank} of a permutation $f$ of $\F_q$, 
which means the smallest value of $n\ge 0$ in any representation of $f$ of the form 
$$f = \theta_0 \circ x^{q-2} \circ \theta_1 \circ x^{q-2} \circ \dots \circ x^{q-2} \circ \theta_n$$
in which each $\theta_i$ is a degree-one polynomial in $\F_q[x]$. For a systematic study of the general theory of Carlitz ranks and their 
generalizations, see \cite{DZcr}. Carlitz ranks have been extensively studied in the literature, see for example the papers 
\cite{ACMT,Ces,CMTcyc,CMTenum,CMTperm,GOT,I,IT,ITW,MT,T}. For the history of permutation polynomials over a finite field, see the papers
\cite{Betti,Carlitz,Carlitz2, Carlitz3,Dicksonthesis,Fryer1,Fryer2,LN,W}.

As the first consequence, our results provide a new perspective on the theory of Carlitz ranks from the point of view of combinatorics, 
instead of from the point of view of algebra as in the literature previously. 

More precisely, by Corollary~\ref{F_q} 
the Carlitz rank of a permutation $f$ of\/ $\F_q$ is exactly the smallest integer $n\ge 0$ for which the uniquely extended permutation
of\/ $\bP^1(\F_q)$, which is also denoted by $f$ by abuse of language, admits a combinatorial representation of the form
$$f = \nu(x) \circ (b_1,\infty) \circ (b_2,\infty) \circ \cdots \circ (b_n,\infty)$$ for some $b_1, b_2, \dots, b_n \in \F_q$ and 
some $\nu(x) \in \F_q(x)$ of degree one. 

Based on this new characterization of Carlitz rank, the following result on the computation of Carlitz rank has be obtained in \cite{DZcr}:

\begin{thm}\label{Crank}
Let $f=\mu\circ\sigma$ be a permutation of\/ $\F_q$, where $\mu$ is a degree-one rational function, and $\sigma$ is a permutation of\/ $\bP^1(\F_q)$ 
which moves $s$ points in\/ $\F_q$ and has $t$ nontrivial orbits. Define $n:=s+t$ if $\sigma(\infty)=\infty$ and $n:=s+t-1$ if $\sigma(\infty)\neq \infty$.
Then the Carlitz rank of $f$ is at most $n$, and it equals $n$ if in addition $q\geq n+s+2$.
\end{thm}

As the second consequence, our results provide new perspective to Theorem~\ref{Carlitz}. Indeed, Theorem~\ref{Carlitz} follows 
easily from Corollary~\ref{F_q} in various ways, since any permutation $f$ of\/ $\F_q$ has many different ways to be represented as 
$$f = (b_1,\infty) \circ (b_2,\infty) \circ \cdots \circ (b_k,\infty)$$ for some $k\ge 1$ and some $b_1, b_2, \dots, b_k \in \F_q$, 
and by Corollary~\ref{F_q} any such combinatorial representation of $f$ gives rise to an algebraic representation of $f$ as 
a composition of $k$ copies of $x^{q-2}$ and degree-one polynomials over\/ $\F_q$, which concludes our proof for Theorem~\ref{Carlitz}. 

More precisely, given any nontrivial permutation $f$ of\/ $\F_q$, we can write $f$ as the product of disjoint cycles of length at least two 
as follows: 
$$f = (b_{11},b_{12},\dots,b_{1s_1}) (b_{21},b_{22},\dots,b_{2s_2})\cdots (b_{t1},b_{t2},\dots,b_{ts_t}),$$
where $t\ge 1$ and each $s_i\ge 2$ with $1\le i\le t$. Note that 
$$(b_{i1},b_{i2},\dots,b_{is_i}) = (b_{i1},\infty) \circ [(b_{is_i},\infty) \circ \cdots \circ (b_{i2},\infty) \circ (b_{i1},\infty)],$$
for any $1\le i\le t$. Hence we get a representation of $f$ of the form
$$f = (b_1,\infty) \circ (b_2,\infty) \circ \cdots \circ (b_n,\infty),$$ where $s:= \Sigma_{i=1}^t s_i$, $k:= s+t$, and denote by 
$(b_1, b_2, \dots, b_k)$ the tuple
$$((b_{11}, b_{1s_1}, \dots, b_{11}), (b_{21}, b_{2s_2}, \dots, b_{21}), \dots, (b_{t1}, b_{ts_t}, \dots, b_{t1})).$$
By Theorem~\ref{second}, if we denote $$(a_1, a_2, \dots, a_k):= G(b_1, b_2, \dots, b_k),$$ then we obtain a representation of $f$ 
of the form $$f = \mu(x) \circ x^{q-2} \circ (x-a_k) \circ x^{q-2} \circ (x-a_{k-1}) \circ \cdots \circ x^{q-2} \circ (x-a_1),$$ 
where $$\mu(x) := (x+a_1) \circ x^{-1} \circ (x+a_2) \circ x^{-1} \circ \cdots \circ (x+a_k) \circ x^{-1} \in \F_q(x)$$ is a rational 
function of degree one. Furthermore, $\mu(x)$ is a degree-one polynomial over $\F_q$ by Corollary~\ref{F_q}. Therefore, we obtain 
an algebraic representation of the given permutation $f$ of\/ $\F_q$ as a composition of $k$ copies of $x^{q-2}$ and $k+1$ degree-one 
polynomials over\/ $\F_q$ as above. Moreover, if in addition $q\ge 2s+t+2$, then this representation of $f$ is optimal in the sense 
that it has the least possible copies of $x^{q-2}$, since in this case the Carlitz rank of $f$ is exactly $k$ by Theorem~\ref{Crank}.

In particular, by application of the above procedure to $2$-cycles of the form $(0,a)$ with $a\in \F_q^*$, we can recover uniformly 
both proofs for Theorem~\ref{Carlitz} in Carlitz \cite{Carlitz} and Zieve \cite{Z}, which involve some clever tricks and have not been 
related to each other previously. 

First, let us review briefly the proofs for Theorem~\ref{Carlitz} in Carlitz \cite{Carlitz} and Zieve \cite{Z}. The starting point for both proofs 
is the same, i.e., it is enough to show the result in the special case that the permutation is a $2$-cycle of the form $(0,a)$ with $a\in \F_q^*$, 
since any permutation of $\F_q$ can be written as a product of such $2$-cycles. But Carlitz \cite{Carlitz} and Zieve \cite{Z} give different 
expressions of the $2$-cycle $(0,a)$ as compositions of $x^{q-2}$ and degree-one polynomials. More precisely, Carlitz \cite{Carlitz} relies 
on the mysterious observation that $$(0,a) = (-a^2x)\circ x^{q-2}\circ (x-a)\circ x^{q-2}\circ (x+1/a)\circ x^{q-2}\circ (x-a),$$
but does not explain how it could be found; Zieve \cite{Z} observes that
$$(0,a) = (-ax+a)\circ x^{q-2}\circ (-x+1)\circ x^{q-2}\circ (-x+1)\circ x^{q-2}\circ (x/a),$$
which follows from the clever combination of Lemma~\ref{basic} and the fact that the degree-one rational function $1-x^{-1}$ induces 
an order-three permutation of\/ $\bP^1(\F_q)$ with a $3$-cycle $(\infty,1,0)$.

Now, we apply our procedure illustrated above to $2$-cycles $(0,a)$ with $a\in \F_q^*$ to recover uniformly both of the above two 
observations in Carlitz \cite{Carlitz} and Zieve \cite{Z} respectively. Indeed, for any $a \in \F_q^*$, there are two equally natural 
ways to express the $2$-cycle $(0,a)$ as a product of some $2$-cycles of the form $(b,\infty)$ with $b\in \F_q$. More precisely, 
by our procedure the $2$-cycle $(0,a)$ can be expressed naturally as either $(0,a) = (a,\infty)\circ (0,\infty)\circ (a,\infty)$ or 
$(0,a) = (0,\infty)\circ (a,\infty)\circ (0,\infty)$. By some computation as explained above, the first identity leads exactly to 
the mysterious observation in Carlitz \cite{Carlitz}, and the second one gives 
$$(0,a) = (-a^2x+a)\circ x^{q-2}\circ (x+a)\circ x^{q-2}\circ (x-1/a)\circ x^{q-2}\circ x,$$ which is the same as the observation 
in Zieve \cite{Z} as polynomials. To see this it is enough to normalize all except the first degree-one polynomials to be monic 
in the observation discovered by Zieve \cite{Z}.

Finally, we emphasize that it is of essential importance to introduce the extra point $\infty$ and to work over $\bP^1(\F_q)$ instead of $\F_q$, 
although both Carlitz's original Theorem~\ref{Carlitz} and the notion of Carlitz rank are only about permutations of\/ $\F_q$ and appear to have 
nothing to do with $\infty$. It is clear that neither of the above two consequences can be possibly obtained by working over $\F_q$ without 
introducing $\infty$.




\begin{thebibliography}{9}
\newcommand{\au}[1]{{#1},}
\newcommand{\ti}[1]{\textit{#1},}
\newcommand{\jo}[1]{{#1}}
\newcommand{\vo}[1]{\textbf{#1}}
\newcommand{\yr}[1]{(#1),}
\newcommand{\ppx}[1]{#1,}
\newcommand{\pp}[1]{#1.}
\newcommand{\pps}[1]{#1}
\newcommand{\bk}[1]{{#1},}
\newcommand{\inbk}[1]{in: \bk{#1}}
\newcommand{\xxx}[1]{{arXiv:#1}.}


\bibitem{ACMT}
\au{E. Aksoy, A. \c Ce\c smelio\u glu, W. Meidl and A. Topuzo\u glu}
\ti{On the Carlitz rank of permutation polynomials}
\jo{Finite Fields Appl.}
\vo{15}
\yr{2009}
\pp{428--440}

\bibitem{Betti}
\au{E. Betti}
\ti{Sopra la risolubilit\`a per radicali delle equazioni algebriche irriduttibili di grado primo}
\jo{Annali di Scienze Matematiche e Fisiche}
\vo{2}
\yr{1851}
\pp{5--19 (=Opere Matematiche, v. 1, 17--27)}

\bibitem{Carlitz}
\au{L. Carlitz}
\ti{Permutations in a finite field}
\jo{Proc. Amer. Math. Soc.}
\vo{4}
\yr{1953}
\pp{538}

\bibitem{Carlitz2}
\au{\bysame}
\ti{Permutations in finite fields}
\jo{Acta Sci. Math. (Szeged)} 
\vo{24} 
\yr{1963}
\pp{196--203}

\bibitem{Carlitz3}
\au{\bysame}
\ti{A note on permutations in an arbitray field}
\jo{Proc. Amer. Math. Soc.}
\vo{14}
\yr{1963}
\pp{101}

\bibitem{Ces}
\au{A. \c Ce\c smelio\u glu}
\ti{A representation of permutations with full cycle}
\xxx{1005.2019v1}

\bibitem{CMTcyc}
\au{A. \c Ce\c smelio\u glu, W. Meidl and A. Topuzo\u glu}
\ti{On the cycle structure of permutation polynomials}
\jo{Finite Fields Appl.}
\vo{14}
\yr{2008}
\pp{593--614}

\bibitem{CMTenum}
\au{\bysame}
\ti{Enumeration of a class of sequences generated by inversions}
\inbk{Coding and Cryptology}
World Sci. Publ., Hackensack, NJ
\yr{2008}
\pp{44--57}

\bibitem{CMTperm}
\au{\bysame}
\ti{Permutations of finite fields with prescribed properties}
\jo{J. Comput. Appl. Math.}
\vo{259}
\yr{2014} part B, 
\pp{536--545}

\bibitem{Dicksonthesis}
\au{L. E. Dickson}
\ti{The analytic representation of substitutions on a power of a prime number of letters, with a discussion of the linear group}
\jo{Annals Math.}
\vo{11}
\yr{1896--1897}
\pp{65--120 and 161--183}

\bibitem{DZcr}
\au{Z. Ding and M. E. Zieve}
\ti{Carlitz ranks and approximate rational functions}
to submit.

\bibitem{Fryer1}
\au{K. D. Fryer}
\ti{A class of permutation groups of prime degree}
\jo{Canad. J. Math.}
\vo{7}
\yr{1955}
\pp{24--34}

\bibitem{Fryer2}
\au{\bysame} 
\ti{Note on permutations in a finite field}
\jo{Proc. Amer. Math. Soc.}
\vo{6}
\yr{1955}
\pp{1--2}

\bibitem{GOT}
\au{D. Gomez-Perez, A. Ostafe and A. Topuzo\u glu}
\ti{On the Carlitz rank of permutations of $\F_q$ and pseudorandom sequences}
\jo{J. Complexity}
\vo{30}
\yr{2014}
\pp{279--289}

\bibitem{I}
\au{L. I\c sik}
\ti{On complete mappings and value sets of polynomials over finite fields}
Ph. D. thesis, Sabanci University, 2015.

\bibitem{IT}
\au{L. I\c sik and A. Topuzo\u glu}
\ti{A note on value sets of polynomials over finite fields}
\xxx{1701.06158v1}

\bibitem{ITW}
\au{L. I\c sik, A. Topuzo\u glu and A. Winterhof}
\ti{Complete mappings and Carlitz rank}
\jo{Des. Codes Cryptogr.}
\vo{85}
\yr{2017}
\pp{121--128}

\bibitem{LN}
\au{R. Lidl and H. Niederreiter}
\bk{Finite Fields}
second ed.,
Encyclopedia Math. Appl.
\vo{20},
Cambridge Univ. Press, New York, 1997.

\bibitem{MT}
\au{W. Meidl and A. Topuzo\u glu}
\ti{On the inversive pseudorandom number generator}
\inbk{Recent developments in applied probability and statistics}
Physica, Heidelberg
\yr{2010}
\pp{103--125}

\bibitem{T}
\au{A. Topuzo\u glu}
\ti{The Carlitz rank of permutations of finite fields: A survey}
\jo{J. Symbolic Comput.}
\vo{64}
\yr{2014}
\pp{53--66}

\bibitem{W}
\au{C. Wells}
\ti{Generators for groups of permutation polynomials over finite fields}
\jo{Acta Sci. Math. Szeged.}
\vo{29}
\yr{1968}
\pp{167--176}

\bibitem{Z}
\au{M. E. Zieve}
\ti{On a theorem of Carlitz}
\jo{J. Group Theory}
\vo{17}
\yr{2014}
\pp{667--669}

\end{thebibliography}
\end{document}